\documentclass[11pt,reqno,letterpaper]{amsart}

\usepackage{booktabs}
\usepackage{color}
\usepackage[colorlinks=true, allcolors=blue,backref=page,hypertexnames=false]{hyperref}
\usepackage{amsmath, amssymb, amsthm}
\usepackage{mathrsfs}
\usepackage{mathtools}
\usepackage[noabbrev,capitalize,nameinlink]{cleveref}
\crefname{equation}{}{}
\usepackage{fullpage}
\usepackage[noadjust]{cite}
\usepackage{graphics}
\usepackage{pifont}
\usepackage{tikz}
\usepackage{bbm}
\usepackage{thm-restate}
\usepackage[T1]{fontenc}

\usetikzlibrary{arrows.meta}

\usepackage{environ}
\usepackage{framed}
\usepackage{url}
\usepackage[linesnumbered,ruled,vlined]{algorithm2e}
\usepackage[noend]{algpseudocode}
\usepackage[labelfont=bf]{caption}
\usepackage{cite}
\usepackage{framed}
\usepackage[framemethod=tikz]{mdframed}
\usepackage{appendix}
\usepackage{graphicx}
\usepackage[textsize=tiny]{todonotes}
\usepackage{tcolorbox}
\usepackage{enumerate}
\allowdisplaybreaks[1]
\usepackage{enumerate}
\usepackage{stmaryrd}

\usepackage[margin=1in]{geometry}

\usepackage[shortlabels]{enumitem}
\crefformat{enumi}{#2#1#3}
\crefrangeformat{enumi}{#3#1#4 to~#5#2#6}
\crefmultiformat{enumi}{#2#1#3}
{ and~#2#1#3}{, #2#1#3}{ and~#2#1#3}

\DeclareSymbolFont{symbolsC}{U}{pxsyc}{m}{n}
\SetSymbolFont{symbolsC}{bold}{U}{pxsyc}{bx}{n}
\DeclareFontSubstitution{U}{pxsyc}{m}{n}
\DeclareMathSymbol{\medcircle}{\mathbin}{symbolsC}{7}

\crefname{algocf}{Algorithm}{Algorithms}

\crefname{equation}{}{} 
\AtBeginEnvironment{appendices}{\crefalias{section}{appendix}} 

\usepackage[color,final]{showkeys} 

\colorlet{refkey}{orange!20}
\colorlet{labelkey}{blue!30}

\crefname{algocf}{Algorithm}{Algorithms}

\newcommand{\R}{\mathbb{R}}

\newcommand{\ip}[2]{\langle #1,#2\rangle}
\newcommand{\norm}[1]{\lVert #1\rVert}

\newcommand{\la}{\lambda_*}

\newcommand{\pos}{\operatorname{pos}}

\numberwithin{equation}{section}
\newtheorem{theorem}{Theorem}[section]
\newtheorem{proposition}[theorem]{Proposition}
\newtheorem{lemma}[theorem]{Lemma}
\newtheorem{claim}[theorem]{Claim}

\crefname{claim}{Claim}{Claims}

\newtheorem{corollary}[theorem]{Corollary}
\newtheorem{conjecture}[theorem]{Conjecture}
\newtheorem*{question*}{Question}

\theoremstyle{definition}
\newtheorem{definition}[theorem]{Definition}

\newtheorem*{definition*}{Definition}

\newtheorem{remark}[theorem]{Remark}

\newcommand{\mb}{\mathbb}

\newcommand{\mc}{\mathcal}
\newcommand{\mf}{\mathfrak}

\newcommand{\on}{\operatorname}

\let\originalleft\left
\let\originalright\right
\renewcommand{\left}{\mathopen{}\mathclose\bgroup\originalleft}
\renewcommand{\right}{\aftergroup\egroup\originalright}

\allowdisplaybreaks

\newcommand{\ignore}[1]{}

\title{Equality in Fill's spectral gap problem}
\author[A1]{Vishesh Jain}
\address{Department of Mathematics, Statistics, and Computer Science, University of Illinois Chicago, Chicago, IL, 60607 USA}
\email{visheshj@uic.edu}

\author[A2]{Clayton Mizgerd}
\address{Department of Mathematics, Statistics, and Computer Science, University of Illinois Chicago, Chicago, IL, 60607 USA}
\email{cmizge2@uic.edu}

\newcommand{\paren}[1]{\left( #1 \right)}

\begin{document}

\begin{abstract}
We study the adjacent-transposition chain on the symmetric group $\mathfrak{S}_n$ with a regular parameter vector $\vec{p} = (p_{i,j})_{i\neq j}$. 
Fill's spectral gap conjecture, recently resolved in the affirmative by Greaves--Zhu, states that among all regular parameter vectors, the spectral gap of the transition matrix is minimized by the uniform vector $p_{i,j}= 1/2$ for all $i\neq j$. 

We prove the stronger statement that among all regular parameter vectors, the spectral gap is minimized if and only if $\vec{p}$ has a neutral label,~i.e., there exists $c \in [n]$ such that $p_{c,i} = 1/2$ for all $i\neq c$. Moreover, in this case, we show that the multiplicity of the second largest eigenvalue is equal to the number of neutral labels, unless the number of neutral labels is $n-2$ or $n$, in which case the multiplicity is $n-1$. This confirms a conjecture of Fill.

\end{abstract}

\maketitle
 

\section{Introduction}

Let $n\geq 2$ be an integer and for $i\neq j$, let $p_{i,j}\in (0,1)$ with $p_{i,j} + p_{j,i} = 1$ for all $i\neq j$. Let $\vec{p} = (p_{i,j})_{1\leq i < j\leq n}$. Fill's \cite{Fill} adjacent-transposition Markov chain on $\mf{S}_n$ with parameter vector $\vec{p}$ can be described as follows. Write the current state $x \in \mf{S}_n$ in one-line notation as $x = (x_1,\dots, x_n)$. Then, the next state is determined by the following rule:

\begin{enumerate}
    \item Choose a position $r \in \{1,\dots, n-1\}$ uniformly at random.
    \item With probability $p_{x_r, x_{r+1}}$, stay at $x$. With the remaining probability $p_{x_{r+1}, x_{r}}$, move to
    \[x^{\tau_r} := (x_1,\dots, x_{r-1}, x_{r+1}, x_{r},\dots, x_{n}).\]
\end{enumerate}

Formally, for $r \in \{1,\dots, n-1\}$, define the $\mf{S}_n \times \mf{S}_n$ matrix $E_r$ by
\[
(E_{r})_{x,y} = 
\begin{cases}
    p_{x_r, x_{r+1}}, \quad &y=x;\\
    p_{x_{r+1}, x_{r}}, \quad &y = x^{\tau_r};\\
    0, \quad &\text{otherwise}.
\end{cases}
\]

Then, the transition matrix of the adjacent-transposition walk is given by
\[K = K(\vec{p}) =  \frac{1}{n-1}\sum_{r=1}^{n-1} E_r.\]

Since $p_{i,j} \in (0,1)$ for all $i\neq j$, the adjacent-transposition walk is ergodic on $\mf{S}_n$. As Fill showed, it is reversible with respect to its stationary distribution $\mu = \mu(\vec{p})$ defined by
\begin{equation}
\label{eqn:stationary-dist}
\mu(x) := Z^{-1} \prod_{1\leq i < j \leq n}p_{x_i, x_j},
\end{equation}
where $Z$ is a normalizing constant chosen so that $\sum_{x} \mu(x) = 1$. 

\medskip

Motivated by a problem on self-organizing lists, Fill considered regular parameter vectors $\vec{p}$, defined as follows. We refer the reader to \cite{Fill} for motivation behind this definition.
\begin{definition}
    The parameter vector $\vec{p} = (p_{i,j})_{1\leq i < j \leq n}$ is said to be regular if:
    \begin{itemize}
        \item  $p_{i-1,i} \ge \frac12  \qquad (2\le i\le n)$ \label{eq:reg1}
 \item $p_{i-1,j} \ge p_{i,j}  \qquad (2\le i<j\le n)$, \label{eq:reg2}
 \item $p_{i,j+1} \ge p_{i,j}  \qquad (1\le i<j\le n-1).$ \label{eq:reg3}
    \end{itemize}
We let $\mc{R}$ denote the set of all regular parameter vectors. 
\end{definition}

For $K = K(\vec{p})$, define the spectral gap $\lambda_K$ of $K$ by
\[\lambda_K = 1-\beta_K,\]
where $\beta_K$ is the second largest eigenvalue of $K$. We remark that since the spectrum of $K$ is contained in $[0,1]$ (see \cref{prop:similarity}), $\lambda_K$ is also the absolute spectral gap of $K$.  

Let $\vec{p}_{\mathrm{unif}}$ denote the uniform parameter vector with all entries equal to $1/2$. Wilson \cite{Wilson04} showed that
\[\lambda_{K(\vec{p}_{\mathrm{unif}})} = \frac{1-\cos(\pi/n)}{n-1} =: \lambda_*\]

Fill made several conjectures about the spectral gap of regular adjacent-transposition walks. Let

\[\Lambda := \{\lambda_{K(\vec{p})}: \vec{p} \in \mc{R}\}.\]

\begin{conjecture}[Gap conjecture]
    \label{conj:gap}
Among all regular choices of $\vec{p}$, the gap $\lambda_{K(\vec{p})}$ is minimized by $\vec{p}_{\mathrm{unif}}$, i.e.~
\[\inf \{\lambda: \lambda \in \Lambda\} = \min \{\lambda: \lambda \in \Lambda\} = \lambda_*.\]
\end{conjecture}
This conjecture was very recently settled in the affirmative by Greaves and Zhu \cite{GZ}; we refer the reader to \cite{GZ} and \cite{Fill} for discussion of prior work. 

\medskip

Fill also made stronger conjectures characterizing all the minimizing regular parameter vectors. In order to state these conjectures, we introduce some terminology. 

\begin{definition}
\label{def:neutral}
Let $\vec{p}$ be a parameter vector. We say that $c \in [n]$ is a neutral label, or simply neutral, if
\[p_{c,i} = p_{i,c} = 1/2 \qquad \text{for all } i\neq c.\]
\end{definition}

\begin{conjecture}[Characterization of minimizers]
    \label{conj:equalitycases}
    Let $n\geq 3$. For a regular parameter vector $\vec{p}$, the following are equivalent:
    \begin{enumerate}[(a)]
        \item $\lambda_{K(\vec{p})} = \lambda_*$
        \item $\vec{p}$ has a neutral label. 
    \end{enumerate}
\end{conjecture}

\begin{remark}
    In the above conjecture, one must restrict to $n=3$, since for $n=2$, a direct computation shows that $\lambda_{K(\vec{p})} = \lambda_*$ for every regular parameter vector $\vec{p}$. On the other hand, a neutral label exists iff $p_{1,2} = 1$. 
\end{remark}

The implication (b) implies (a) is straightforward and follows from Wilson's argument (see \cref{prop:neutral-position-eigenfunction}), so that the main content is proving (a) implies (b). 

\medskip

In the conjectured equality cases, Fill also made a conjecture about the multiplicity of the second largest eigenvalue. Once again, we introduce some notation to state this conveniently. 

\begin{definition}
    \label{def:multiplicity}
For a parameter vector $\vec{p}$, let \[N(\vec{p}) := |\{i\in [n]: i \text{ is neutral}\}|
\]
denote the number of neutral labels. Let
\[M(\vec{p}) := \text{(algebraic) multiplicity of } 1-\lambda_* \text{ in }\on{spec}(K(\vec{p})).\]
\end{definition}

\begin{conjecture}[Multiplicity conjecture]
    \label{conj:multiplicity}
    Let $n\geq 3$. 
Let $\vec{p}$ be a regular parameter vector. The following relationship holds between $N(\vec{p})$ and $M(\vec{p})$:
\begin{enumerate}[(i)]
    \item If $N(\vec{p}) \notin \{n-2, n\}$, then
    \[N(\vec{p}) = M(\vec{p}).\]
    \item Otherwise, if $N(\vec{p}) \in \{n-2, n\}$, then $M(\vec{p}) = n-1$.
\end{enumerate}
\end{conjecture}

We confirm both of these conjectures. 

\begin{theorem}
\label{thm:main}
Let $\vec{p}$ be a regular parameter vector. Then,
\begin{enumerate}
\item \[\lambda_{K(\vec{p})} = \lambda_* \iff N(\vec{p}) \geq 1.\]
\item Moreover,
\[
M(\vec{p})
=
\begin{cases}
    N(\vec{p}), \qquad &\emph{if }N(\vec{p}) \notin \{n-2, n\};\\
    n-1, \qquad &\emph{otherwise}.
\end{cases}
\]
\end{enumerate}
\end{theorem}

\subsection*{Overview of the proof} The starting point of our proof is the Greaves--Zhu proof of Fill's spectral gap conjecture. Let $K = K(\vec{p})$ have spectral gap $\lambda_K$. This implies that the smallest positive eigenvalue of $L := I-K$ is $\lambda_K$. In \cite{GZ}, they further use the similarity of $K$ and $L$ to conclude that the smallest positive eigenvalue of $K$ is $\lambda_K$. The main work in their paper is to show that for regular $\vec{p}$, $\lambda_K \geq \lambda_*$ by comparing the quadratic forms $\langle f, K^2 f \rangle$ and $\langle f, K f \rangle$ for arbitrary $f \in \mb{R}^{\mf{S}_n}$ and where the inner product is the weighted inner product induced by the stationary measure.  Instead of shifting to $K$, we will find it more convenient to work directly with $L$. The argument from \cite{GZ} applies essentially verbatim for $L$ and shows that $\lambda_L \geq \lambda^*$ (see \cref{prop:Lineq}).

Suppose $\lambda_L = \lambda^*$ and let $0\neq g \in \mb{R}^{\mf{S}_n}$ be an eigenvector of $L$ with eigenvalue $\lambda_*$. Then, every inequality in the the Greaves--Zhu argument for $L$ and $g$ must be an equality. This allows us to deduce several rigidity properties of the eigenvector. In particular, consider the orthogonal projection $F_1 := I - E_1$ and let $u_1 := F_1 g$. The equality conditions show that $u_1(x)$ depends only on the first two labels of $x$. This allows us to completely describe $u_1$, which is \emph{a priori} an $n!$-dimensional vector, by a family $(U_{a,b})_{a< b}$ of $\binom{n}{2}$ real numbers (see \cref{eq:def-U}).   

The core of the paper is to analyze the behavior of $u_1$ on the $G_1 := \langle \tau_1, \tau_2 \rangle \simeq \mf{S}_3$ orbits of $\mf{S}_n$. Each such orbit is completely specific by three labels $i < j < k$. Once again, by equality in Greaves--Zhu, we show that whenever $u_1$ is nonzero on an orbit, then one must have  
\[
p_{i,j}=p_{j,k}=\frac12,
\]
together with a linear relation among the corresponding coefficients
$U_{i,j},U_{i,k},U_{j,k}$ (see \cref{cor:triplesystem}). 

The existence of a neutral label is then obtained by combining this orbital relation with the monotonicity built into regular parameter vectors. Consider the set of pairs $(a,b)$ for which $U_{a,b}\neq 0$. Using the orbital relation, we show that this set meets the triangle $1\leq a < b \leq n$ at both the left and top edges in a certain manner (see \cref{lem:boundarypairs,lem:AB}) This gives a label $c$ with $p_{1,c} = p_{c,n} = 1/2$ at which point regularity forces $c$ to be neutral.

The multiplicity statement is obtained by pushing the same structure further. Regularity shows that the neutral labels form an interval $[A,B]$ (see \cref{lem:neutral-interval-structure}) and the orbital relation then shows that
the $\binom{n}{2}$ coefficients $(U_{a,b})_{a<b}$ are, in fact, determined by the adjacent values
\[
 D_r:=U_{r,r+1}.
 \]
 for $r \in [A-1, B]$ (see \cref{lem:boundary-data-structure}). This essentially gives the desired upper bound on the multiplicity (except for a possible shift by one, see \cref{prop:upper-bound-multiplicity}). For the lower bound, each neutral label contributes Wilson's single-card eigenfunction, so the only remaining case is the exceptional regime $N(\vec p)=n-2$. There we construct an additional linearly independent eigenfunction built from the relative positions of the two non-neutral labels (see \cref{prop:special-extra-eigenfunction}).

\subsection*{Acknowledgments} Upon being given the complete proof of \cref{thm:main}(1), the proof of \cref{thm:main}(2) was produced by ChatGPT 5.4 Pro in a ``one-shot'' manner and subsequently verified for correctness by the authors. The proof we include here (see \cref{sec:multiplicity}) is the same proof, but rewritten by the authors for clarity. V.J.~is partially supported by NSF grant DMS-2237646. C.M.~is partially supported by a Simons Dissertation Fellowship.

\section{Preliminaries}
Throughout the remainder of this paper, $n\geq 3$. Our proof builds on the approach of Greaves--Zhu \cite{GZ}. As in \cite{GZ}, we equip $\mb{R}^{\mf{S}_n}$ with the weighted inner product
\[\langle f, g\rangle = \langle f, g \rangle_{\vec{p}} := \sum_{x \in \mf{S}_n} f(x)g(x)\mu(x),\]
where recall that $\mu = \mu_{\vec{p}}$ is the stationary distribution from \cref{eqn:stationary-dist}. This inner product induces the norm
\[\|f\| = \|f\|_{\vec{p}} := \sqrt{\langle f, f \rangle}.\]
We remark that throughout we will omit the dependencies of various quantities on the parameter vector $\vec{p}$ when there is no risk of confusion.

\begin{proposition}[{\cite[Theorem~2]{Fill}}]
\label{prop:similarity}
  For any parameter vector $\vec{p}$, the matrices $K$ and $I-K$ are similar. In particular, the spectrum of $K$ lies in $[0,1]$ and is symmetric around $1/2$. Furthermore, $0$ and $1$ are eigenvalues of $K$ with multiplicity $1$.  
\end{proposition}

Henceforth, for convenience of notation, we set
\[L := I-K. \]

Since $K$ and $L$ have the same spectrum by \cref{prop:similarity}, we may (and will) equivalently prove \cref{thm:main} for $L$, instead of $K$. Accordingly, for $r \in \{1,\dots, n-1\}$, let
\[F_r := I - E_r\]
and notice that
\[L = \frac{1}{n-1}\sum_{r=1}^{n-1}F_r.\]

Define the subspaces $V_r, W_r \subset \mb{R}^{\mf{S}_n}$ by
\[V_r := \{f \in \mb{R}^{\mf{S}_n}: f(x^{\tau_r}) = f(x) \quad \forall x \in \mf{S}_n \} \]
and
\[W_r = V_r^{\perp},\]
where recall that we are working with the weighted inner product $\langle \cdot, \cdot \rangle$. We will make use of the following lemmas from \cite{GZ}. 

\begin{lemma}[{\cite[Lemma~2.3]{GZ}}]
For each $1\leq r \leq n-1$, $E_r$ is the orthogonal projection onto $V_r$  with respect to $\langle \cdot, \cdot \rangle$. 

Consequently, $F_r$ is the orthogonal projection onto $W_r$ with respect to the same inner product. 
\end{lemma}

\begin{lemma}[{\cite[Lemma~2.1]{GZ}}]
\label{lem:commute}
Let $r,s \in \{1,\dots, n-1\}$. Suppose $|r-s| > 1$. Then, $E_r E_s = E_s E_r$ and consequently, $F_r F_s = F_s F_r$.     
\end{lemma}

Next, we will need a refinement of \cite[Lemma~2.2]{GZ}. To state it, we need to introduce some notation, which will also be useful later. 

\medskip
Let $r \in \{1,\dots, n-2\}$. Recall the transposition operators $\tau_r, \tau_{r+1}$ and note that the subgroup $G_r$ generated by $\tau_r$ and $\tau_{r+1}$ is isomorphic to $\mf{S}_3$. Since $G_r$ acts only on positions $r, r+1, r+2$, each $G_r$-orbit is the six-point set obtained by permuting the three labels in those positions while leaving all other positions fixed. 

Let $T_r$ be a set of orbit representatives. For $t \in T_r$, let $\mc{O}_t := G_r t$ denote the corresponding orbit, and let $V_{\mc{O}_t}$ denote the linear span of $\{e_x: x \in \mc{O}_t\}$, where $e_x$ denotes the elementary basis vector in $\mb{R}^{\mf{S}_n}$ corresponding to $x \in \mf{S}_n$. This gives the decomposition 
\begin{equation}
\label{eqn:orbit-decomposition}
\mb{R}^{\mf{S}_n} = \oplus_{t \in T_r} V_{\mc{O}_t},
\end{equation}
and each subspace $V_{\mc{O}_t}$ is invariant under $E_{r} + E_{r+1}$ and $F_{r} + F_{r+1}$. 

Note that each orbit $\mc{O}$ can be described by three distinct labels $i < j < k$ such that positions $r, r+1, r+2$ for every element in the orbit are a permutation of $(i,j,k)$. For such an orbit $\mc{O}$, we define
\[
s_{\mc{O}}:=p_{i,j}p_{j,k}p_{k,i}+p_{k,j}p_{j,i}p_{i,k}, \qquad m_{\mc{O}} = 1-\sqrt{s_{\mc{O}}}.
\]

\begin{lemma}[{\cite[Lemma~2.2]{GZ}}]
\label{lem:orbit-matrix}
Let $r\in\{1,\dots,n-2\}$. Let $\mc{O}$ be a $G_r$-orbit in $\mf{S}_n$. Let $t$ be the representative in $\mc{O}$ with
\[
t_r=i,\qquad t_{r+1}=j,\qquad t_{r+2}=k.
\]
for $i < j < k$. Then, in the ordered basis
\[
t,\ t^{\tau_r},\ t^{\tau_{r+1}},\ t^{\tau_{r+1}\tau_r},\ t^{\tau_r\tau_{r+1}},\ t^{\tau_{r+1}\tau_r\tau_{r+1}},
\]
the restriction of $M_r:=E_r+E_{r+1}$ to $\R^{\mc{O}}$ is
\[
M_{\mc{O}}=
\begin{pmatrix}
p_{i,j}+p_{j,k} & p_{j,i} & p_{k,j} & 0 & 0 & 0\\
p_{i,j} & p_{j,i}+p_{i,k} & 0 & p_{k,i} & 0 & 0\\
p_{j,k} & 0 & p_{i,k}+p_{k,j} & 0 & p_{k,i} & 0\\
0 & p_{i,k} & 0 & p_{j,k}+p_{k,i} & 0 & p_{k,j}\\
0 & 0 & p_{i,k} & 0 & p_{k,i}+p_{i,j} & p_{j,i}\\
0 & 0 & 0 & p_{j,k} & p_{i,j} & p_{k,j}+p_{j,i}
\end{pmatrix}.
\]
\end{lemma}

A direct computation shows the following. 

\begin{corollary}
\label{lem:blockspec}
With notation as above, 
the characteristic polynomial of $M_{\mc{O}}$ is
\[
\chi_{\mc{O}}(\alpha)=\alpha(\alpha-2)(\alpha^2-2\alpha+1-s_{\mc{O}})^2.
\]
Therefore, the eigenvalues of $M_{\mc{O}}$ are
\[
0,\qquad 2,\qquad 1-\sqrt{s_{\mc{O}}},\qquad 1+\sqrt{s_{\mc{O}}},
\]
with multiplicities $1,1,2,2$, respectively. 

Consequently the restriction $F_r+F_{r+1}=2I-(E_r+E_{r+1})$ to $\R^{\mc{O}}$ has the same multiset of eigenvalues.
\end{corollary}

The next lemma extends the key linear-algebraic \cite[Lemma~2.4]{GZ} by deriving an important relationship in the equality case.  

\begin{lemma}\label{lem:PQP}
Let \(0 \le \gamma < 1\), and let \(P,Q\) be orthogonal projections on a finite-dimensional real inner product space \(V\). Assume that every nonzero eigenvalue of \(P+Q\) is at least \(1-\gamma\). Then, for every $f \in V$,
\[\langle Pf, Qf \rangle \geq -\gamma \norm{Pf}\norm{Qf}.\]
Moreover, if
\[
\langle Pf,Qf\rangle=-\gamma\|Pf\|\,\|Qf\|\neq 0,
\]
then
\[
PQPf=\gamma^2Pf.
\]
Equivalently, \(Pf\) is an eigenvector of \(PQ\) with eigenvalue \(\gamma^2\).
\end{lemma}

\begin{proof}
The first part is exactly \cite[Lemma~2.4]{GZ}. The ``moreover'' part is new and we will prove this here.  

\medskip

Let
\[
W:=\operatorname{Im}(P)\cap \operatorname{Im}(Q).
\]
Since \(P\) and \(Q\) fix \(W\) pointwise, the orthogonal complement \(W^\perp\) is invariant under both \(P\) and \(Q\).

We first record the operator-norm consequence of the spectral assumption.

\medskip
\begin{claim}\label{claim:Wperp-is-nice}
For every \(x\in W^\perp\),
\[
\|PQx\|\le \gamma\,\|Qx\|, \qquad \|QPx\| \leq \gamma \|Px\|.
\]
\end{claim}

\begin{proof}
We will prove the first case; the second case holds by symmetry.  Let
\[
U:=\operatorname{Im}(Q)\cap W^\perp,
\qquad
A:=QPQ|_U.
\]
Then \(A\) is self-adjoint and positive semidefinite on \(U\). Let \(\sigma^2\) be a nonzero eigenvalue of \(A\), and choose a unit eigenvector \(y\in U\) such that
\[
Ay=\sigma^2 y.
\]
Then \(0<\sigma\le 1\). In fact \(\sigma<1\): if \(\sigma=1\), then
\[
\|Py\|^2=\langle y,QPy\rangle=\langle y,Ay\rangle=1=\|y\|^2,
\]
so \(y\in\operatorname{Im}(P)\). Since \(y\in\operatorname{Im}(Q)\) as well, this would imply \(y\in W\), contradicting \(y\in W^\perp\).

Now define
\[
z:=\sigma^{-1}Py.
\]
Then \(z\in \operatorname{Im}(P)\cap W^\perp\), and
\[
Qz=\sigma^{-1}QPy=\sigma^{-1}Ay=\sigma y.
\]
Hence
\[
(P+Q)(z-y)=Pz-Py+Qz-Qy
=z-\sigma z+\sigma y-y
=(1-\sigma)(z-y).
\]
Also \(z-y\neq 0\), because otherwise \(y=z\in \operatorname{Im}(P)\cap \operatorname{Im}(Q)=W\), again impossible. Therefore \(1-\sigma\) is a nonzero eigenvalue of \(P+Q\). By hypothesis,
\[
1-\sigma\ge 1-\gamma,
\]
so \(\sigma\le \gamma\).

Thus every eigenvalue of \(A\) is at most \(\gamma^2\). Now take any \(x\in W^\perp\), and set \(y:=Qx\in U\). Then
\[
\|PQx\|^2=\|Py\|^2=\langle y,Ay\rangle\le \gamma^2\|y\|^2=\gamma^2\|Qx\|^2. \qedhere
\]
\end{proof}

Now decompose
\(
f=w+u
\)
with \(w\in W\) and \(u\in W^\perp\).  We will show $w = 0$.  Notice
\[
\langle Pf,Qf\rangle=\|w\|^2+\langle Pu,Qu\rangle.
\]

By \cref{claim:Wperp-is-nice},
\[
|\langle Pu,Qu\rangle|
=
|\langle Pu,PQu\rangle|
\le \|Pu\|\,\|PQu\|
\le \gamma \|Pu\|\|Qu\|,
\]
so
\begin{equation}\label{eq:PfQfw}
\langle Pf,Qf\rangle\ge \|w\|^2-\gamma\|Pu\|\|Qu\|.
\end{equation}
Also,
\begin{equation}\label{eq:PfQfcxy}
\|Pf\|\,\|Qf\|
=
\sqrt{(\|w\|^2+\|Pu\|^2)(\|w\|^2+\|Qu\|^2)}
\ge \|w\|^2+\|Pu\| \|Qu\|,
\end{equation}
because for any $c,x,y \in \mb{R}_{\geq0}$,
\[
(c+x^2)(c+y^2)-(c+xy)^2=c(x-y)^2\ge 0.
\]

Suppose \(w \ne 0 \). Then combining \cref{eq:PfQfw,eq:PfQfcxy},
\[
\langle Pf,Qf\rangle
\ge \|w\|^2 -\gamma\|Pu\| \|Qu\|
> -\gamma(\|w\|^2+ \|Pu\|\|Qu\|)
\ge -\gamma\|Pf\|\,\|Qf\|,
\]
by \cref{eq:PfQfcxy}, contradicting the hypothesis
\[
\langle Pf,Qf\rangle=-\gamma\|Pf\|\,\|Qf\|.
\]
Therefore \(f\in W^\perp\).

Since \(\langle Pf,Qf\rangle\neq 0\), both \(Pf\) and \(Qf\) are nonzero. Applying \cref{claim:Wperp-is-nice} to \(f\in W^\perp\), we get
\[
\|PQf\|\le \gamma\,\|Qf\|.
\]
Therefore
\[
\gamma\|Pf\|\,\|Qf\|
=
-\langle Pf,Qf\rangle
=
|\langle Pf,PQf\rangle|
\le
\|Pf\|\,\|PQf\|
\le
\gamma\|Pf\|\,\|Qf\|.
\]
Hence equality holds throughout. Equality in Cauchy--Schwarz implies that \(PQf\) is a scalar multiple of \(Pf\), and
\[
\langle Pf,PQf\rangle=-\gamma\|Pf\|\,\|Qf\| \implies PQf = -\gamma \frac{\|Qf\|}{\|Pf\|} Pf.
\]

By the same argument with $P$ and $Q$ interchanged,
\[
QPf=-\gamma\frac{\|Pf\|}{\|Qf\|}\,Qf.
\]
Consequently,
\[
PQPf=P(QPf)=\paren{-\gamma\frac{\|Pf\|}{\|Qf\|}} \,P(Qf)=\paren{-\gamma \frac{\|Qf\|}{\|Pf\|}} \paren{-\gamma\frac{\|Pf\|}{\|Qf\|}}\,Pf=\gamma^2Pf. \qedhere
\]
\end{proof}

\section{Revisiting the Greaves--Zhu argument}
For a parameter vector $\vec{p}$, let
\[
 m_{\vec{p}}:=\max_{1\le i<j<k\le n}\sqrt{p_{i,j}p_{j,k}p_{k,i}+p_{k,j}p_{j,i}p_{i,k}}.
\]

\begin{lemma}[{\cite[Lemma~3.3]{GZ}}]
\label{lem:mp-bound}
    Let $\vec{p}$ be regular. Then, $m_{\vec{p}}\leq 1/2$. 
\end{lemma}

The following is the analogue of \cite[Proposition~3.1]{GZ} for $K$ and follows using exactly the same argument. We record the argument below for the reader's convenience since we will need to isolate the equality conditions.

\begin{proposition}\label{prop:Lineq}
Let $f\in \R^{\mf{S}_n}$. Then,
\[
 \ip{f}{L^2f}
 \ge
 \frac{1-2m_{\vec{p}}\cos(\pi/n)}{n-1}\,\ip{f}{Lf}.
\]
\end{proposition}

\begin{proof}
First, if $|r-s|>1$, then $F_rF_s=F_sF_r$ by \cref{lem:commute}. Since $F_r, F_s$ are orthogonal projections, it follows that $F_r F_s$ is idempotent and self-adjoint. Therefore,
\[
 \langle{F_rf},{F_sf}\rangle=\langle{f},{F_rF_sf}\rangle=\norm{F_rF_sf}^2\ge 0.
\]
Next, let $1\le r\le n-2$. By \cref{lem:blockspec}, on each $G_r$-orbit $\mc{O}$, every nonzero eigenvalue of $F_r+F_{r+1}$ is at least $1-m_{\vec{p}}$.

Therefore, applying \cref{lem:PQP} with $P=F_r$ and $Q=F_{r+1}$ shows that for all $1\leq r \leq n-2$,
\[
 \langle{F_rf},{F_{r+1}f}\rangle \ge -m_{\vec{p}}\norm{F_rf}\norm{F_{r+1}f}.
\]
Therefore
\begin{align*}
 \langle{f},{L^2f}\rangle
 &= \frac{1}{(n-1)^2}
 \Bigg(
 \sum_{r=1}^{n-1}\norm{F_rf}^2
 +2\sum_{r=1}^{n-2}\langle{F_rf},{F_{r+1}f}\rangle
 +2\sum_{\substack{1\le r<s\le n-1\\ |r-s|>1}}\ip{F_rf}{F_sf}
 \Bigg)\\
 &\ge
 \frac{1}{(n-1)^2}
 \Bigg(
 \sum_{r=1}^{n-1}\norm{F_rf}^2
 -2m_{\vec{p}}\sum_{r=1}^{n-2}\norm{F_rf}\norm{F_{r+1}f}
 \Bigg).
\end{align*}
Set $v_r:=\norm{F_rf}$ and $\vec{v}=(v_1,\dots,v_{n-1})^T \in \mb{R}^{n-1}$. The right-hand side is $v^T T v$, where
\[
 T = T_{\vec{p}}:=\frac{1}{(n-1)^2}
 \begin{pmatrix}
 1 & -m_{\vec{p}} \\
 -m_{\vec{p}} & 1 & -m_{\vec{p}} \\
 & \ddots & \ddots & \ddots \\
 && -m_{\vec{p}} & 1 & -m_{\vec{p}} \\
 &&& -m_{\vec{p}} & 1
 \end{pmatrix}.
\]
The eigenvalues of the unscaled tridiagonal matrix are $1-2m_{\vec{p}}\cos(k\pi/n)$, $1\le k\le n-1$ (see \cite[Section 1.4.4]{brouwer2011spectra}); hence the smallest eigenvalue of $T$ is
\[
 \frac{1-2m_{\vec{p}}\cos(\pi/n)}{(n-1)^2}.
\]
Thus
\[
 \langle{f},{L^2f}\rangle
 \ge
 \frac{1-2m_{\vec{p}}\cos(\pi/n)}{(n-1)^2}
 \sum_{r=1}^{n-1}\norm{F_rf}^2.
\]
Finally, since each $F_r$ is an orthogonal projection,
\[
 \sum_{r=1}^{n-1}\norm{F_rf}^2
 =
 \sum_{r=1}^{n-1}\langle{f},{F_rf}\rangle
 =(n-1)\langle{f},{Lf}\rangle. \qedhere
\]
\end{proof}

Let $\lambda_L$ denote the smallest positive eigenvalue of $L$. We now examine the consequences of  $\lambda_L=\lambda_*$. 
Choose a nonzero eigenvector $g \in \mb{R}^{\mf{S}_n}$ such that
\[ Lg=\lambda_* g.\]

Since $\vec{p}$ is regular, \cref{prop:Lineq} and \cref{lem:mp-bound} give
\[
 \lambda_*
 =\frac{\ip{g}{L^2g}}{\ip{g}{Lg}}
 \ge \frac{1-2m_{\vec{p}}\cos(\pi/n)}{n-1}
 \ge \frac{1-\cos(\pi/n)}{n-1}
 =\lambda_*.
\]
Hence both inequalities must be equalities.  In particular,
\begin{equation}\label{eq:mphalf}
 m_{\vec{p}}=\frac12.
\end{equation}
and equality holds throughout the proof of \cref{prop:Lineq} for the vector $g$.

\begin{lemma}\label{lem:eqcase}
Let $\vec{p}$ be regular and
let $0 \neq g \in \mb{R}^{\mf{S}_n}$ satisfy
\[Lg = \lambda_* g.\]
For $r \in \{1,\dots, n-1\}$, define
\[
 u_r:=F_rg,
 \qquad
 v_r:=\norm{u_r}.
\]
The following hold.
\begin{itemize}
\item[(a)] If $|r-s|>1$, then $F_rF_sg=0$.
\item[(b)] The vector $(v_1,\dots,v_{n-1})$ is a positive multiple of
\[
 (\sin(\pi/n),\sin(2\pi/n),\dots,\sin((n-1)\pi/n)).
\]
In particular, $u_r\ne 0$ for every $r$.
\item[(c)] For every $1\le r\le n-2$,
\[
 \langle{u_r},{u_{r+1}}\rangle=-\frac12\norm{u_r}\norm{u_{r+1}}.
\]
\end{itemize}
\end{lemma}

\begin{proof}
The proof of \cref{prop:Lineq} yields
\begin{align*}
 \langle{g},{L^2g}\rangle
 &= \frac{1}{(n-1)^2}
 \Bigg(
 \sum_{r=1}^{n-1}\norm{u_r}^2
 +2\sum_{r=1}^{n-2}\langle{u_r},{u_{r+1}}\rangle
 +2\sum_{\substack{1\le r<s\le n-1\\ |r-s|>1}}\langle{u_r},{u_s}\rangle
 \Bigg)\\
 &\ge
 \frac{1}{(n-1)^2}
 \Bigg(
 \sum_{r=1}^{n-1}v_r^2-2m_{\vec{p}}\sum_{r=1}^{n-2}v_rv_{r+1}
 \Bigg)
 = v^TT_{\vec{p}}v
 \ge \lambda_{\min}(T_{\vec{p}})\sum_{r=1}^{n-1}v_r^2,
\end{align*}
where $v=(v_1,\dots,v_{n-1})^T$.  Because equality holds throughout and $m_{\vec{p}}=1/2$, every step above is an equality.

For $|r-s|>1$, we have
\[
 \langle{u_r},{u_s}\rangle=\norm{F_rF_sg}^2\ge 0.
\]
Since the sum of these nonnegative terms is zero, each one vanishes.  This proves (a).

Equality in the Rayleigh bound forces $v$ to lie in the one-dimensional eigenspace of $T$ corresponding to its smallest eigenvalue.  That eigenspace is spanned by 
\[
 (\sin(\pi/n),\sin(2\pi/n),\dots,\sin((n-1)\pi/n)),
\]
all of whose entries are positive.
Since $v_r\ge 0$ and $v\ne 0$, statement (b) follows.

Finally,
\[
 \sum_{r=1}^{n-2}\Bigl(\langle{u_r},{u_{r+1}}\rangle+\tfrac12 v_rv_{r+1}\Bigr)=0.
\]
Each summand is nonnegative, because \cref{lem:PQP} gives
\[
 \ip{u_r}{u_{r+1}}\ge -m_{\vec{p}}v_rv_{r+1}=-\tfrac12 v_rv_{r+1}.
\]
Hence every summand is zero, proving (c).
\end{proof}

\section{The key orbital relation}
\label{sec:orbital}
Throughout this section, let $\vec{p}$ be regular and let $0\neq g \in \mb{R}^{\mf{S}_n}$ satisfy
\[Lg = \lambda_* g.\]
Recall that this forces $m_{\vec{p}} = 1/2$ and \cref{lem:eqcase}. In particular, by \cref{lem:eqcase}(a), for every $s\ge 3$,
\[
 F_su_1=F_sF_1g=0.
\]
Since $F_s=I-E_s$, this means $E_su_1=u_1$ for every $s\ge 3$.  Thus $u_1$ is invariant under all adjacent transpositions in positions $3,4,\dots,n$, so $u_1$ depends only on the first two labels. 

\begin{remark}
\label{rmk:LvsK}
This characterization of $u_1$ is why we chose to work with $F$ and $L$ instead of $E$ and $K$. Note also that the choice of $u_1$ is important for this invariance statement (although we could have worked with $u_{n-1}$ as well).  
\end{remark}

For $1\le a<b\le n$, define $U_{a,b}$ by
\begin{equation}
\label{eq:def-U}
 U_{a,b}:=u_1(x)\qquad\text{whenever }x_1=a,\ x_2=b.
\end{equation}
Because $u_1\ne 0$, the family $U=(U_{a,b})_{1\leq a < b \leq n}$ is not identically zero.  

Further, since $F_1$ is idempotent, $F_1u_1=u_1$, equivalently $E_1u_1=0$. Hence, if $x$ is any permutation whose first two labels are $a,b$ with $a<b$ and $x^{\tau_1}$ is obtained by swapping those two labels, then
\[u_1(x^{\tau_1}) = -\frac{p_{a,b}}{p_{b,a}}u_1(x) = -\frac{p_{a,b}}{p_{b,a}}U_{a,b}.\]
Thus, the values of $u_1$ on states whose first two labels are $b,a$ are also determined by $U_{a,b}$. 

\medskip

The main goal of this section is to prove the following linear relation on every orbit which supports $u_1$. In the statement below, let $i<j<k$, and let $\omega$ be any ordering of the remaining labels
$[n]\setminus\{i,j,k\}$.  
Consider the corresponding $G_1$-orbit
\[
\mc{O}(i,j,k;\omega):=
\{(i,j,k,\omega),\ (j,i,k,\omega),\ (i,k,j,\omega),\ (j,k,i,\omega),\ (k,i,j,\omega),\ (k,j,i,\omega)\}.
\]
Because $u_1$ depends only on the first two labels, its restriction to $\mc O(i,j,k;\omega)$ is
completely determined by the three coefficients $U_{i,j},U_{i,k},U_{j,k}$ and the swapped
values forced by $E_1u_1=0$; explicitly,
\[
u_1|_{\mc{O}(i,j,k;\omega)}=
\left(U_{i,j},\,-\frac{p_{i,j}}{p_{j,i}}U_{i,j},\,U_{i,k},\,U_{j,k},\,-\frac{p_{i,k}}{p_{k,i}}U_{i,k},\,-\frac{p_{j,k}}{p_{k,j}}U_{j,k}\right).
\]
In particular, $u_1|_{\mc{O}(i,j,k;\omega)}\not\equiv 0$ if and only if at least one of
$U_{i,j},U_{i,k},U_{j,k}$ is nonzero.

\begin{proposition}\label{cor:triplesystem}
Fix $i < j < k$ and an ordering $\omega$ of the remaining labels $[n]\setminus \{i,j,k\}$. If $u_1$ is not identically zero on the orbit $\mc{O}(i,j,k;\omega)$, then
\[
p_{i,j}=p_{j,k}=\frac12
\]
and
\[
U_{i,k}=2p_{k,i}(U_{i,j}+U_{j,k}).
\]
\end{proposition}

The proof of this proposition relies on the following lemma. Recall the decomposition
\[\mb{R}^{\mf{S}_n} = \oplus_{t \in T_1} V_{\mc{O}_t}\]
from \cref{eqn:orbit-decomposition}. Corresponding to this, decompose
\[g = \sum_{t \in T_1} g_t,\]
where $g_t$ is supported only on $\mc{O}_t$. In particular, for $t\neq s$, $\langle g_s, g_t \rangle = 0$.

\begin{lemma}\label{lem:orbit-rigidity}
For $t \in T_1$, let $i < j < k$ be the three labels appearing in positions $1, 2, 3$ on $\mc{O}_t$. If $\norm{F_1 g_t} > 0$, then the following hold.
\begin{itemize}
\item $\langle F_1 g_t, F_{2} g_t \rangle  = -(1/2) \norm{F_1 g_t} \norm {F_{2} g_t}$;
\item $s_{\mc{O}_t} = 1/4$;
\item $p_{i,j}=p_{j,k}=\frac12$.
\end{itemize}
\end{lemma}

\begin{proof}
Because the supports of the $g_t$ are disjoint and both $F_1$ and $F_2$ preserve each
orbit $\mc{O}_t$, we have orthogonal decompositions
\[
u_1=\sum_{t\in T_1}F_1g_t,
\qquad
u_2=\sum_{t\in T_1}F_2g_t.
\]
Hence
\begin{align*}
-\frac12\|u_1\|\,\|u_2\|
&=\langle u_1,u_2\rangle =\sum_{t\in T_1}\langle F_1g_t,F_2g_t\rangle \\
&\ge -\sum_{t\in T_1} \sqrt{s_{\mc{O}_t}} \|F_1 g_t\| \|F_2 g_t \| \\
&\ge -\frac12\sum_{t\in T_1} \|F_1 g_t\| \|F_2 g_t \| \ge -\frac12
   \Bigl(\sum_{t\in T_1} \|F_1 g_t\|^2\Bigr)^{1/2}
   \Bigl(\sum_{t\in T_1} \|F_2 g_t \|^2\Bigr)^{1/2} = -\frac12\|u_1\|\,\|u_2\|.
\end{align*}
The first line uses \cref{lem:eqcase}(c). The second line is \cref{lem:PQP} applied to each real inner product space $\mb{R}^{\mc{O}_t}$, using \cref{lem:blockspec} to
see that every nonzero eigenvalue of $(F_1+F_2)|_{\R^{\mc{O}_t}}$ is at least $1-\sqrt{s_{\mc{O}_t}}$.
The last line uses $\sqrt{s_{\mc{O}_t}}\le 1/2$, which follows from \cref{lem:mp-bound}, followed by Cauchy--Schwarz.

Since the first and final terms are equal, all inequalities must be equalities, immediately showing the first two bullets.

It remains to deduce $p_{i,j}=p_{j,k}=1/2$. By definition,
\[
s_{\mc{O}_t}
= p_{i,j}p_{j,k}(1-p_{i,k}) + (1-p_{i,j})(1-p_{j,k})p_{i,k}
= p_{i,j}p_{j,k} + p_{i,k}\bigl(1-p_{i,j}-p_{j,k}\bigr).
\]

Since $\vec{p}$ is regular and $i<j<k$, we have $1/2 \leq p_{i,j},p_{j,k}$ and so $1-p_{i,j}-p_{j,k} \leq 0$.  Thus by monotonicity,
\[
s_{\mc{O}_t}
\le p_{i,j}p_{j,k} + p_{j,k}\bigl(1-p_{i,j}-p_{j,k}\bigr)
= p_{j,k}(1-p_{j,k})
\le \frac14,
\]
and also
\[
s_{\mc{O}_t}
\le p_{i,j}p_{j,k} + p_{i,j}\bigl(1-p_{i,j}-p_{j,k}\bigr)
= p_{i,j}(1-p_{i,j})
\le \frac14.
\]

As we have already shown $s_{\mc{O}_t} = 1/4$, these are equalities.  Noting that $x(1-x)=1/4$ occurs only at $x=1/2$ completes the proof.
\end{proof}

We can now prove the main result of this section.

\begin{proof}[Proof of \cref{cor:triplesystem}]
Let $t \in T_1$ be the representative of the orbit $\mc{O} = \mc{O}(i,j,k;\omega) = \mc{O}_t$. Suppose that $u_1$ is not identically zero on the orbit. By definition, this means that $\|F_1 g_{t}\| > 0$, so that by \cref{lem:orbit-rigidity},
\[p_{i,j} = p_{j,k} = \frac 12.\]
It remains to prove the second conclusion.

By \cref{lem:orbit-rigidity}, 
\[
\langle F_1g_t,F_2g_t\rangle=-\frac12\|F_1g_t\|\,\|F_2g_t\|\ne 0.
\]
Moreover, since $s_{\mc{O}} = 1/4$ by \cref{lem:orbit-rigidity}, it follows from \cref{lem:blockspec} that the 
nonzero eigenvalues of $(F_1+F_2)|_{\R^{\mc{O}}}$ are $1/2,3/2,2$. Hence \cref{lem:PQP}
applies on $\R^{\mc{O}}$ and yields
\[
F_1F_2F_1g_t=\frac14 F_1 g_t.
\]

In the ordered basis
\[
(i,j,k,\omega),\ (j,i,k,\omega),\ (i,k,j,\omega),\ (j,k,i,\omega),\ (k,i,j,\omega),\ (k,j,i,\omega)
\]
of $\R^{\mc{O}}$, the restrictions of $F_1$ and $F_2$ are
{\setlength{\arraycolsep}{4pt}
\[
F_1=
\begin{pmatrix}
\frac12&-\frac12&0&0&0&0\\
-\frac12&\frac12&0&0&0&0\\
0&0&p_{k,i}&0&-p_{k,i}&0\\
0&0&0&\frac12&0&-\frac12\\
0&0&-p_{i,k}&0&p_{i,k}&0\\
0&0&0&-\frac12&0&\frac12
\end{pmatrix},
\qquad
F_2=
\begin{pmatrix}
\frac12&0&-\frac12&0&0&0\\
0&p_{k,i}&0&-p_{k,i}&0&0\\
-\frac12&0&\frac12&0&0&0\\
0&-p_{i,k}&0&p_{i,k}&0&0\\
0&0&0&0&\frac12&-\frac12\\
0&0&0&0&-\frac12&\frac12
\end{pmatrix}.
\]
}
Moreover, since $p_{i,j} = p_{j,i} = p_{j,k} = p_{k,j} = 1/2$, then in the same basis, we have
\[F_1 g_t = (U_{i,j}, -U_{i,j}, U_{i,k}, U_{j,k}, -(p_{i,k}/p_{k,i})U_{i,k}, -U_{j,k})\]
Then, a direct computation shows
\[
F_1F_2F_1 g_t -\frac14 F_1 g_t
=\frac{U_{i,k}-2p_{k,i}(U_{i,j}+U_{j,k})}{4}
\left(-1,1,1,-\frac{p_{i,k}}{p_{k,i}},-\frac{p_{i,k}}{p_{k,i}},\frac{p_{i,k}}{p_{k,i}}\right).
\]
Therefore
\[
U_{i,k}=2p_{k,i}(U_{i,j} + U_{j,k}),
\]
as desired. 
\end{proof}

\section{Existence of a neutral label}
We can now prove the first part of \cref{thm:main}; this confirms \cref{conj:equalitycases}. Recall that $\vec{p}$ is a regular parameter vector satisfying
\[\lambda_K = \lambda_L = \lambda_*.\]
Recall also that $g \in \mb{R}^{\mf{S}_n}$ is a nonzero vector satisfying $Lg = \lambda_* g$ and $u_1 = F_1 g$. For $1\leq a < b\leq n$, $U_{a,b}$ are defined in \cref{eq:def-U}.

Define
\[
S:=\{(a,b):1\le a<b\le n,\ U_{a,b}\ne 0\}.
\]
Thus $(a,b)\in S$ means exactly that $u_1$ is nonzero on the class of permutations whose first
two labels are $a,b$. The set $S$ is nonempty because $u_1\ne 0$.

We need two structural facts about $S$.

\begin{lemma}\label{lem:boundarypairs}
There exist indices $b\in\{2,\dots,n\}$ and $a\in\{1,\dots,n-1\}$ such that
\[
(1,b)\in S,\qquad (a,n)\in S.
\]
\end{lemma}

\begin{proof}
We prove the left-boundary statement first. Choose $(a,b)\in S$ with $a$ minimal. Then
$U_{a,b}\ne 0$.

Suppose for contradiction that $a>1$. Consider the triple $a-1<a<b$. Since one of the three
coefficients attached to this triple, namely $U_{a,b}$, is nonzero, \cref{cor:triplesystem}
applies and gives
\[
U_{a-1,b}=2p_{b,a-1}(U_{a-1,a}+U_{a,b}).
\]
By the minimality of $a$, this becomes
\(
0=2p_{b,a-1}U_{a,b},
\)
a contradiction as $0 < p_{b,a-1}$ and $U_{a,b} \ne 0$ by assumption.

The proof of the right-boundary statement is analogous by regarding the triple $a<b<b+1$.
\end{proof}

Now define
\[
B:=\max\{b:(1,b)\in S\},\qquad A:=\min\{a:(a,n)\in S\}.
\]
These are nonempty sets by \cref{lem:boundarypairs}.

\begin{lemma}\label{lem:AB}
With notation as above, $A\le B$.
\end{lemma}

\begin{proof}
If $A = 1$, then $A\leq B$, so assume $A > 1$. By \cref{cor:triplesystem} on the triple $1 < A < n$ (since $U_{A,n} \ne 0$), we have
\[
U_{1,n}=2p_{n,1}(U_{1,A}+U_{A,n}).
\]

If $A > B$, then $U_{1,A} = U_{1,n} = 0$, giving
\[
0=2p_{n,1}U_{A,n}.
\]
Again $0<p_{n,1}<1$ and $U_{A,n}\ne 0$, a contradiction. Therefore $A\le B$.
\end{proof}

With these ingredients, we can proceed to the proof of \cref{thm:main}(1).

\begin{proof}[Proof of \cref{thm:main}(1)]
By \cref{lem:AB} the interval
\[
I:=\bigl[\max\{2,A\},\,\min\{B,n-1\}\bigr]\cap \mathbb Z
\]
is nonempty. Choose any integer $c\in I$. 
We claim that
\[
p_{1,c}=\frac12,\qquad p_{c,n}=\frac12.
\]

To prove the first equality, we need only find some triple with one nonzero attached coefficient and cite \cref{cor:triplesystem}.

If $c<B$, then the triple $1<c<B$ has $U_{1,B} \ne 0$ by maximality of $B$.  If $c=B$, then $(1,c)\in S$, so the triple $1<c<n$ has the nonzero attached coefficient
$U_{1,c} \ne 0$.  Thus $p_{1,c} = 1/2$ by \cref{cor:triplesystem}.

The proof of $p_{c,n}=1/2$ is similar, now using $(A,n)\in S$.
Apply \cref{cor:triplesystem} to the triple $A<c<n$ (if $c > A$) or $1<c<n$ (if $c=A$).

We now propagate these equalities to full neutrality of the label $c$ using the monotonicity conditions. 
Regularity of $\vec{p}$ implies
$p_{i,j}\ge 1/2$ for every $i<j$; indeed, for fixed $i<j$ one has $p_{i,j}\ge p_{i,i+1}\ge 1/2$.
Also, by monotonicity in the first coordinate, for $i < c$,
\[
p_{i,c}\le p_{1,c}=\frac12,
\]
and by monotonicity in the second coordinate, for $j > c$,
\[
p_{c,j}\le p_{c,n}=\frac12.
\]
Combining these inequalities with the lower bound $p_{i,j}\ge 1/2$ gives $p_{c,j} = 1/2$ for all $j\neq c$,
so $c$ is a neutral label.
\end{proof}

\section{Multiplicity of the second largest eigenvalue}
\label{sec:multiplicity}
In this section, we prove \cref{thm:main}(2), thereby confirming \cref{conj:multiplicity}. Recall that $M(\vec{p})$ is defined to be the algebraic multiplicity of $1-\lambda_*$ as an eigenvalue of $K$. By \cref{prop:similarity}, $M(\vec{p})$ is equal to the algebraic multiplicity of $\lambda_*$ as an eigenvalue of $L$. Accordingly, let
\[
\mc{E}_*:=\ker(L-\lambda_*I).
\]
Since $L$ is self-adjoint with respect to $\langle \cdot,\cdot\rangle$, $\mc{E}_*$ has dimension equal to the algebraic multiplicity of $\lambda_*$ as an eigenvalue of $L$, so that
\[
M(\vec{p})=\dim \mc{E}_*.
\]

If $N(\vec{p})=0$, then \cref{thm:main}(1) gives $\lambda_L  >\lambda_*$, so $M(\vec{p}) = 0$. So, it remains to prove the multiplicity conjecture for $N(\vec{p})\ge 1$. We will work under this assumption for the rest of this section. 

\medskip

\subsection{Upper bound} We first prove the upper bound in \cref{thm:main}(2). Fix a regular parameter vector $\vec{p}$ with $N(\vec{p})\geq 1$. Define
\[
A:=\min\{c\in [n]: c \text{ is neutral}\},
\quad
B:=\max\{c\in [n]: c \text{ is neutral}\},
\]
so that
\[N(\vec{p}) = B-A+1.\]

The labels $A$ and $B$ divide the interval $[n]$ into three (possibly empty) regions: $i < A$, $i \in [A,B]$, $i > B$. 

\begin{definition}
\label{def:crossing-pair}
    We say that the pair $1\leq i < j \leq n$ is crossing if $i < A$ and $j > B$. Otherwise, we say that the pair is non-crossing. 
\end{definition}

The following lemma records some elementary consequences of regularity. 

\begin{lemma}
\label{lem:neutral-interval-structure}
The following hold.
\begin{enumerate}[(i)]
\item Every label in $[A,B]$ is neutral.
\item If $A=1$ or $B=n$, then $\vec{p}=\vec{p}_{\mathrm{unif}}$.
\item If $\vec{p}\neq \vec{p}_{\mathrm{unif}}$, so that $2\le A\le B\le n-1$, then
\[
p_{i,j}=\frac12
\qquad\text{whenever } 1\le i<j\le n \text{ is non-crossing.}
\]
In particular, for $i<A$, the label $i$ is non-neutral if and only if there exists $k>B$ with $p_{i,k}>1/2$; likewise, for $j>B$, the label $j$ is non-neutral if and only if there exists $k<A$ with $p_{k,j}>1/2$.
\end{enumerate}
\end{lemma}

\begin{proof}
For (i), the endpoints $A$ and $B$ are neutral by definition, so fix $c\in (A,B)$.

For any $i < c$, by regularity, $p_{i,c} \geq 1/2$.  By monotonicity in the second coordinate, $p_{i,c} \leq p_{i,B} = 1/2$, so $p_{i,c} = 1/2$.

Similarly, for any $j > c$, by regularity and monotonicity, $1/2 \leq p_{c,j} \leq p_{A,j} = 1/2$.  Thus $c$ is neutral.

For (ii), suppose first that $A=1$, so label $1$ is neutral. Then, for every $1\le i<j\le n$, monotonicity in the first coordinate gives
\(
p_{i,j}\le p_{1,j}=\frac12.
\)
Regularity again gives $p_{i,j}\ge 1/2$, so $p_{i,j}=1/2$ for all $i<j$. Hence $\vec{p}=\vec{p}_{\mathrm{unif}}$. The case $B=n$ is analogous, using monotonicity in the second coordinate.

Now assume $\vec{p}\neq \vec{p}_{\mathrm{unif}}$. By (ii), we then have $2\le A\le B\le n-1$. To prove (iii), let $i<j$ be a non-crossing pair. If $j\le B$, then $B$ is neutral, so monotonicity in the second coordinate gives
\[
p_{i,j}\le p_{i,B}=\frac12.
\]
If instead $i\ge A$, then $A$ is neutral, so monotonicity in the first coordinate gives
\[
p_{i,j}\le p_{A,j}=\frac12.
\]
In either case, regularity forces $p_{i,j}=1/2$. 
\end{proof}

As in  \cref{sec:orbital}, fix $0\neq g\in \mc{E}_*$ and let $u_1=F_1g$. Recall $U_{a,b}$ from \cref{eq:def-U}. For convenience of notation, for $r\in \{1,\dots, n-1\}$, set
\[
D_r :=U_{r,r+1}
\]


The next lemma shows that the entire family $(U_{a,b})_{1\leq a<b\leq n}$ is determined by the values $(D_r)_{1\leq r\leq n-1}$.

\begin{lemma}
\label{lem:boundary-data-structure}
Assume $\vec{p} \neq \vec{p}_{\mathrm{unif}}$. The following hold.
\begin{enumerate}[(i)]
\item If $1\leq i<j<A$ or $B<i<j \leq n$, then $U_{i,j}=0$.
\item If $r \notin [A-1,B]$, then $D_r = 0$.
\item If $1\leq i < j \leq n$ is non-crossing, then \[U_{i,j}=\sum_{r=i}^{j-1} D_r = \sum_{r \in [i,j-1]\cap[A-1,B]} D_r.
\]
\item If $1\leq i < j \leq n$ is crossing, then 
\[
U_{i,j}=2p_{j,i}\sum_{r=i}^{j-1} D_r = 2p_{j,i}\sum_{r = A-1}^{B}D_r.
\]
\end{enumerate}

\end{lemma}

\begin{proof}
For (i), suppose first that $i<j<A$. Since $j$ is non-neutral, \cref{lem:neutral-interval-structure}(iii) yields some $k>B$ with $p_{j,k}>1/2$. If one of $U_{i,j},U_{i,k},U_{j,k}$ were nonzero, then \cref{cor:triplesystem} applied to the triple $i<j<k$ would force $p_{j,k}=1/2$, a contradiction. Hence, $U_{i,j}=0$. The case $B<i<j$ is analogous.

Statement (ii) is immediate from (i) applied to adjacent pairs.

For (iii), we argue by induction on $j-i$. If $j=i+1$, then
$
U_{i,j}=U_{i,i+1}=D_i
$
by definition. Now let $j-i\ge 2$, and assume the claim holds for all smaller gaps. Since $(i,j)$ is
non-crossing, the pairs $(i,j-1)$, $(j-1,j)$ are also non-crossing. Hence
\cref{lem:neutral-interval-structure}(iii) gives
\[
p_{i,j-1}=p_{j-1,j}=p_{i,j}=\frac12.
\]
By the induction hypothesis,
\[
S:=U_{i,j-1}+D_{j-1}=\sum_{r=i}^{j-1} D_r.
\]
If $S=0$, then if $U_{i,j}\neq 0$, \cref{cor:triplesystem} applied to the triple
$i<j-1<j$ yields
\[
U_{i,j}=2p_{j,i}\bigl(U_{i,j-1}+U_{j-1,j}\bigr)
      =U_{i,j-1}+D_{j-1}
      =S
      =0,
\]
a contradiction. Hence $U_{i,j}=0=S$.

If instead $S\neq 0$, then at least one of $U_{i,j-1}$ and $D_{j-1}$ is nonzero, so
\cref{cor:triplesystem} applied to $i<j-1<j$ gives
\[
U_{i,j}=2p_{j,i}\bigl(U_{i,j-1}+U_{j-1,j}\bigr)
      =U_{i,j-1}+D_{j-1}
      =S
      =\sum_{r=i}^{j-1} D_r.
\]
This proves the first equality in~(iii). The second follows from part~(ii).

For (iv), let $1\le i<j\le n$ be crossing, and fix any $c\in [A,B]$. Then both pairs
$(i,c)$ and $(c,j)$ are non-crossing. Hence part~(iii) gives
\[
U_{i,c}=\sum_{r=i}^{c-1} D_r,
\qquad
U_{c,j}=\sum_{r=c}^{j-1} D_r.
\]
Therefore
\[
T:=U_{i,c}+U_{c,j}=\sum_{r=i}^{j-1} D_r.
\]

If $T=0$, then if $U_{i,j}\neq 0$, \cref{cor:triplesystem} applied to the triple
$i<c<j$ gives
\[
U_{i,j}=2p_{j,i}\bigl(U_{i,c}+U_{c,j}\bigr)=2p_{j,i}T=0,
\]
a contradiction. Hence $U_{i,j}=0=2p_{j,i}T$, and the required identity follows.

Assume now that $T\neq 0$. Then at least one of $U_{i,c}$ and $U_{c,j}$ is nonzero, so
\cref{cor:triplesystem} applied to $i<c<j$ yields
\[
U_{i,j}=2p_{j,i}\bigl(U_{i,c}+U_{c,j}\bigr)
      =2p_{j,i}T
      =2p_{j,i}\sum_{r=i}^{j-1} D_r.
\]
This proves the first equality in~(iv). The second follows from part~(ii).
\end{proof}

\begin{corollary}
\label{cor:d-map-injective}
Assume $\vec{p}\neq \vec{p}_{\mathrm{unif}}$. Then the linear map
\[
\mc{E}_*\to \R^{B-A+2},
\qquad
g\mapsto (D_{A-1}(g),\dots,D_B(g)),
\]
is injective.
\end{corollary}

\begin{proof}
If all of $D_{A-1}(g),\dots,D_B(g)$ vanish, then \cref{lem:boundary-data-structure} shows that $U_{i,j}(g)=0$ for every $1\leq i<j\leq n$. Hence $u_1=F_1g=0$, and therefore $g=0$ by \cref{lem:eqcase}(b).
\end{proof}

We can now obtain the upper bound in \cref{thm:main}(2). 

\begin{proposition}
\label{prop:upper-bound-multiplicity}
Let $\vec{p}$ be regular. Then
\[
M(\vec{p})\le
\begin{cases}
N(\vec{p}), &\text{if } N(\vec{p})\notin \{n,n-2\},\\
 n-1, &\text{if } N(\vec{p})\in \{n,n-2\}.
\end{cases}
\]
\end{proposition}

\begin{proof}
If $N(\vec{p}) = n$, i.e.~if $\vec{p}=\vec{p}_{\mathrm{unif}}$, then the same argument as \cref{lem:boundary-data-structure}(iii) shows that for all $1\leq i < j\leq n$,
\[
U_{i,j}=\sum_{r=i}^{j-1} D_r.
\]
Therefore, the same argument as in \cref{cor:d-map-injective} shows that
\[
M(\vec{p})=\dim \mc{E}_*\le n-1.
\]

We may therefore assume that $\vec{p}\neq \vec{p}_{\mathrm{unif}}$, so that $2\le A\le B\le n-1$. By \cref{cor:d-map-injective},
\[
M(\vec{p})=\dim \mc{E}_*\le B-A+2=N(\vec{p})+1 \neq n-1.
\]
This proves the upper bound in the case $N(\vec{p}) = n-2$, so it remains to treat the case $(A,B)\neq (2,n-1)$, where we want to gain an additional dimension over the bound above. 

For this, it suffices to show that every $g\in \mc{E}_*$ satisfies
\begin{equation}
\label{eq:sum-d-zero}
\sum_{r=A-1}^B D_r(g)=0.
\end{equation}
so that the image of $\mc{E}_*$ under the linear map in \cref{cor:d-map-injective} is contained in a linear space of dimension $B-A+1$.

To prove \cref{eq:sum-d-zero}, fix $g\in \mc{E}_*$. 
If $A>2$, then $A-1$ is non-neutral and greater than $1$. By \cref{lem:neutral-interval-structure}(iii), there exists $k>B$ with $p_{A-1,k}>1/2$, and monotonicity in the second coordinate gives
\[
p_{A-1,n}\ge p_{A-1,k}>\frac12.
\]
Hence, for the triple $1<A-1<n$, \cref{cor:triplesystem} implies that
\[
U_{1,n}(g)=0.
\]
If instead $A=2$, then our assumption $(A,B)\neq (2,n-1)$ forces $B<n-1$. In that case $B+1$ is non-neutral, and because $1$ is the only label to its left outside the neutral interval, \cref{lem:neutral-interval-structure}(iii) yields
\[
p_{1,B+1}>\frac12.
\]
Applying \cref{cor:triplesystem} to the triple $1<B+1<n$ again gives
\[
U_{1,n}(g)=0.
\]
Finally, \cref{lem:boundary-data-structure}(iv) with the crossing pair $(i,j)=(1,n)$ gives
\[
0=U_{1,n}(g)=2p_{n,1}\sum_{r=A-1}^{B} D_r(g).
\]
This proves \cref{eq:sum-d-zero} and completes the proof.
\end{proof}

\subsection{Lower bound}

We now prove the matching lower bounds for the multiplicity. The argument
splits into two parts. First, each neutral label $c$ gives rise to an eigenfunction $f_c(x)$, which is the standard single-card eigenfunction for the adjacent-transposition
shuffle, going back to Wilson \cite{Wilson04}, and in the present setting it applies whenever
$c$ is neutral. These eigenfunctions already give the desired lower bound in
all cases except $N(\vec{p})=n-2$. Thus the only remaining case is $N(\vec{p})=n-2$, where one needs an additional eigenfunction (\cref{prop:special-extra-eigenfunction}). This part requires a new idea. 

\medskip

Throughout, let
\[
h(r):=\cos\paren{\frac{(r-\tfrac12)\pi}{n}}
\qquad (1\le r\le n), \qquad a_0:=\frac{h(1)-h(2)}{2}.
\]
Since $h(1)>h(2)$, one has $a_0\ne 0$.

\begin{proposition}
\label{prop:neutral-position-eigenfunction}
Assume that $c\in [n]$ is neutral, and define $f_c:\mf{S}_n\to \R$ by
\[
f_c(x):=h(\pos_x(c)).
\]
Then, $
Lf_c=\la f_c.
$
Moreover, for every $1\le r\le n-1$,
\[
D_r(f_c)=a_0\paren{\mathbf{1}_{\{r=c\}}-\mathbf{1}_{\{r=c-1\}}}.
\]
\end{proposition}

\begin{proof}
Extend $h$ to $\{0,1,\dots,n+1\}$ by setting
\[
h(0):=h(1),
\qquad
h(n+1):=h(n).
\]
Then, by direct computation,
\[
h(r-1)+h(r+1)=2\cos(\pi/n)h(r)
\qquad (1\le r\le n).
\]
Let $x\in \mf{S}_n$, and write $r:=\pos_x(c)$. Since $c$ is neutral, its
position evolves as the lazy nearest-neighbour walk on $[n]$ that moves one
step left or right with probability $1/(2(n-1))$ whenever that move is
available. Therefore
\[
(Kf_c)(x)=\paren{1-\frac{1}{n-1}}h(r)+\frac{h(r-1)+h(r+1)}{2(n-1)}
        =(1-\lambda_*)h(r)
        =(1-\lambda_*)f_c(x),
\]
so $Lf_c=\la f_c$.

We next compute $D_r(f_c)$. 
Let $x\in\mf{S}_n$. If neither of the first two labels of $x$ is $c$, then swapping
those labels does not change $\pos_x(c)$. Thus
\[
(F_1f_c)(x)=f_c(x)-E_1f_c(x)=0.
\]
Therefore, if $r\notin\{c-1,c\}$, 
\[
D_r(f_c)=U_{r,r+1}(f_c)=0.
\]

Now suppose that the first two labels are $c,b$ with $c<b$. Then $c$ is in position
$1$ in $x$ and in position $2$ in $x^{\tau_1}$. Since $c$ is neutral, we have
\begin{align*}
U_{c,b}(f_c) &=(F_1f_c)(x)  =f_c(x)-\paren{p_{c,b}f_c(x)+p_{b,c}f_c(x^{\tau_1})} \\
  &=h(1)-\frac12 h(1)-\frac12 h(2) = a_0. 
\end{align*}
Taking $b=c+1$, we obtain $
D_c(f_c)=U_{c,c+1}(f_c)=a_0.
$

Similarly, if the first two labels are $a,c$ with $a<c$, then $c$ is in position
$2$ in $x$ and in position $1$ in $x^{\tau_1}$. By a similar computation as above, we have
\begin{align*}
U_{a,c}(f_c)
  =f_c(x)-\paren{p_{a,c}f_c(x)+p_{c,a}f_c(x^{\tau_1})} 
  =h(2)-\frac12 h(2)-\frac12 h(1) = -a_0.
\end{align*}
Taking $a=c-1$, we obtain
$
D_{c-1}(f_c)=U_{c-1,c}(f_c)=-a_0 
$ as claimed. 
\end{proof}

\begin{corollary}
\label{cor:neutral-label-lower-bound}
Let $\vec{p}$ be regular.
\begin{enumerate}[(i)]
\item If $\vec{p}\ne \vec{p}_{\mathrm{unif}}$, then the family $\{f_c:A\le c\le B\}$ is linearly independent. Consequently,
\[
M(\vec{p})\ge N(\vec{p}).
\]
\item If $\vec{p}=\vec{p}_{\mathrm{unif}}$, then the family $\{f_c:1\le c\le n\}$ has rank $n-1$. Consequently,
\[
M(\vec{p})\ge n-1.
\]
\end{enumerate}
\end{corollary}

\begin{proof}
Let $D(f_c) := (D_1(f_c),\dots, D_{n-1}(f_c)) \in \mb{R}^{n-1}$. 
By \cref{prop:neutral-position-eigenfunction},
\[
D(f_c)=a_0(e_c-e_{c-1}),
\]
where $e_0=e_n:=0$ and $e_1,\dots,e_{n-1}$ are the standard basis vectors of
$\R^{n-1}$.

If $\vec{p}\ne \vec{p}_{\mathrm{unif}}$, then $2\le A\le B\le n-1$, and the
vectors
\[
e_A-e_{A-1},\ e_{A+1}-e_A,\ \dots,\ e_B-e_{B-1}
\]
are linearly independent. Hence so are $f_A,\dots,f_B$, proving~(i).
If $\vec{p}=\vec{p}_{\mathrm{unif}}$, then the vectors $e_c-e_{c-1}$ for $1\leq c \leq n$ span $\mb{R}^{n-1}$, proving~(ii).
\end{proof}

It remains only to treat the exceptional case $N(\vec{p})=n-2$, where we need to produce an additional linearly independent eigenvector. 

\begin{proposition}
\label{prop:special-extra-eigenfunction}
Assume $N(\vec{p})=n-2$, so that the neutral interval is $[2,n-1]$. 
Define $\psi:\mf{S}_n\to \R$ by
\[
\psi(x):=
\begin{cases}
 h(\pos_x(1))-h(\pos_x(n)), &\text{if } \pos_x(1)<\pos_x(n),\\[1ex]
 \frac{p_{1,n}}{p_{n,1}}\paren{h(\pos_x(1))-h(\pos_x(n))}, &\text{if } \pos_x(n)<\pos_x(1).
\end{cases}
\]
Then, $
L\psi=\la\psi$. Moreover,
\[
D_1(\psi)=a_0,
\qquad
D_{n-1}(\psi)=\frac{p_{1,n}}{p_{n,1}} a_0,
\qquad
D_r(\psi)=0\quad (2\le r\le n-2).
\]
In particular, $\psi$ is linearly independent of $f_2,\dots,f_{n-1}$.
\end{proposition}

\begin{proof}

For convenience of notation, we set
\[q:=p_{n,1}\in \paren{0,\frac12}, \qquad \rho := \frac{1-q}{q}.\]

First, we verify that $K\psi(x) = (1-\lambda_*)\psi(x)$. Fix $x\in\mf{S}_n$, and write
\[
a:=\pos_x(1),
\qquad
b:=\pos_x(n).
\]

Assume first that $|a-b|>1$. Then the labels $1$ and $n$ are not compared in one step
from $x$, so their relative order cannot change after one step. Hence, on the entire
one-step neighbourhood $y$ of $x$, the function $\psi(y)$ is either
\[
h(\pos_y(1))-h(\pos_y(n))
\qquad\text{or}\qquad
\rho\paren{h(\pos_y(1))-h(\pos_y(n))},
\]
with the same choice holding throughout that neighbourhood. Moreover, every one-step
comparison involving the label $1$ or the label $n$ is with a neutral label, so the computation from \cref{prop:neutral-position-eigenfunction} applies
separately to the two terms. Therefore, $K\psi(x)=(1-\lambda_*)\psi(x)$ in this case.

It remains to treat the case $|a-b| = 1$. Suppose first that $b=a+1$. Note that $E_r \psi(x) = \psi(x)$, unless $r \in \{a-1, a, a+1\}$. For the remaining cases, by direct computation and by neutrality of labels $\{2,\dots, n-1\}$, we have
\begin{itemize}
    \item $E_{a} \psi(x) = (1-q)\psi(x)+q\rho\paren{h(a+1)-h(a)}=0.$
    \item $E_{a-1} \psi(x) =\frac12\psi(x)+\frac12\paren{h(a-1)-h(a+1)}  =\psi(x)+\frac12\paren{h(a-1)-h(a)}$.
    \item $E_{a+1} \psi(x) =\frac12\psi(x)+\frac12\paren{h(a)-h(a+2)} =\psi(x)+\frac12\paren{h(a+1)-h(a+2)}.$
\end{itemize}




Putting these cases together,
\[
(n-1)K\psi(x)
=(n-2)\psi(x)+\frac12\paren{h(a-1)-h(a)}+\frac12\paren{h(a+1)-h(a+2)}.
\]
Moreover, 
\begin{align*}
\paren{h(a-1)-h(a)}+\paren{h(a+1)-h(a+2)}
&=\paren{h(a-1)+h(a+1)}-\paren{h(a)+h(a+2)} \\
&=2\cos(\pi/n)\paren{h(a)-h(a+1)} \\
&=2\cos(\pi/n)\,\psi(x).
\end{align*}
Therefore
\[
(n-1)K\psi(x)
=\paren{n-2+\cos(\pi/n)}\psi(x)
=(n-1)(1-\la)\psi(x).
\]
A similar argument also shows that $K\psi(x) = (1-\lambda_*)\psi(x)$ if $a = b+1$. 


Finally, we compute $D_r(\psi)$ for $1\leq r \leq n-1$. If $2\le r\le n-2$, then the first two labels are $r,r+1$, so swapping them does not
change the positions of $1$ and $n$, nor their relative order. Hence
\[
D_r(\psi)=0
\qquad (2\le r\le n-2).
\]
Now suppose that the first two labels of $x$ are $1,2$, and write $b:=\pos_x(n)>2$. Then the
label $1$ lies to the left of $n$ both before and after swapping the first two labels.
Since the label $2$ is neutral, we have
\begin{align*}
D_1(\psi)
&=(F_1\psi)(x) =\psi(x)-\paren{p_{1,2}\psi(x)+p_{2,1}\psi(x^{\tau_1})} \\
&=\paren{h(1)-h(b)}-\frac12\paren{h(1)-h(b)}-\frac12\paren{h(2)-h(b)} \\
&=\frac12\paren{h(1)-h(2)} =a_0.
\end{align*}

Similarly, suppose that the first two labels of $x$ are $n-1,n$, and write $a:=\pos_x(1)>2$.
Then the label $n$ lies to the left of $1$ both before and after swapping the first two
labels. Since the label $n-1$ is neutral, we obtain
\begin{align*}
D_{n-1}(\psi)
&=(F_1\psi)(x) =\psi(x)-\paren{p_{n-1,n}\psi(x)+p_{n,n-1}\psi(x^{\tau_1})} \\
&=\rho\paren{h(a)-h(2)}-\frac12\rho\paren{h(a)-h(2)}-\frac12\rho\paren{h(a)-h(1)} \\
&=\frac{\rho}{2}\paren{h(1)-h(2)} =\rho a_0.
\end{align*}

Finally, every linear combination of $D(f_2),\dots,D(f_{n-1})$ has coordinate sum zero,
since each $D(f_c)$ is a multiple of $e_c-e_{c-1}$. On the other hand,
\[
\sum_{r=1}^{n-1} D_r(\psi)=a_0+\rho a_0\ne 0.
\]
Hence $D(\psi)$ does not lie in the span of $D(f_2),\dots,D(f_{n-1})$, and therefore
$\psi$ is linearly independent of $f_2,\dots,f_{n-1}$.
\end{proof}

\bibliographystyle{alpha}
\bibliography{main}

\end{document}